\documentstyle[amssymb,amsfonts,12pt]{amsart}

\newtheorem{theorem}{Theorem}[section]
\newtheorem{lemma}[theorem]{Lemma}
\newtheorem{corollary}[theorem]{Corollary}
\newtheorem{proposition}[theorem]{Proposition}
\newtheorem{definition}[theorem]{Definition}

\theoremstyle{remark}

\newtheorem*{ack*}{Acknowledgment}

\def\x{{\bf x}}

\def\R{{\mathbb R}}
\def\E{{\mathcal H}}
\def\N{{\mathbb N}}
\def\C{{\mathbb C}}

\def\P{{\mathcal P}}
\def\B{{\mathbb E}}

\def\A{{\mathcal N}}
\def\dist{{\operatorname{dist}}}
\def\supp{{\operatorname{supp}}}
\def\bas{\begin{align*}}
\def\eas{\end{align*}}
\def\bi{\begin{itemize}}
\def\ei{\end{itemize}}
\newenvironment{proof}{\noindent {\bf Proof} }{\endprf\par}
\def \endprf{\hfill  {\vrule height6pt width6pt depth0pt}\medskip}
\def\1{{\bf 1}}

\begin{document}

\title[$l^p$ decouplings for hypersurfaces with nonzero Gaussian curvature ]{$l^p$ decouplings for hypersurfaces with nonzero Gaussian curvature}
\author{Jean Bourgain}
\address{School of Mathematics, Institute for Advanced Study, Princeton, NJ 08540}
\email{bourgain@@math.ias.edu}
\author{Ciprian Demeter}
\address{Department of Mathematics, Indiana University, 831 East 3rd St., Bloomington IN 47405}
\email{demeterc@@indiana.edu}

\thanks{The first author is partially supported by the NSF grant DMS-1301619. The second  author is partially supported  by the NSF Grant DMS-1161752}
\begin{abstract}
We prove an $l^p$ decoupling inequality for hypersurfaces with nonzero Gaussian curvature and use it to derive a corresponding $l^p$ decoupling for curves not contained in a hyperplane. This extends our earlier work from \cite{BD3}.
\end{abstract}
\maketitle


\section{Statements of results}

Let $S$ be a compact $C^2$ hypersurface in $\R^n$  with  nonzero Gaussian curvature.  The typical example to have in mind is the truncated  paraboloid defined for $\upsilon=(\upsilon_1,\ldots,\upsilon_{n-1})\in(\R\setminus\{0\})^{n-1}$ as
$$H^{n-1}_\upsilon:=\{(\xi_1,\ldots,\xi_{n-1},\upsilon_1\xi_1^2+\ldots+\upsilon_{n-1}\xi_{n-1}^2):\;|\xi_i|\le 1/2\}.$$
 $H^{n-1}_\upsilon$ is called elliptic when all $\upsilon_i$ have the same sign and  hyperbolic otherwise.

Let $\A_\delta$ be the $\delta$ neighborhood of $S$ and let $\P_\delta$ be a finitely overlapping cover of  $\A_\delta$ with
$\sim\delta^{1/2}\times \ldots\delta^{1/2}\times \delta$ rectangular boxes $\theta$. We will denote by $f_\theta$ the Fourier restriction of $f$ to $\theta$.

We will write $A\sim B$ if $A\lesssim B$ and $B\lesssim A$. The implicit constants hidden inside the symbols  $\lesssim$ and $\sim$  will in general  depend on  fixed parameters such as $p$, $n$, $\alpha$ and sometimes on  variable parameters such as  $\epsilon$. We will in general not record the dependence on the fixed parameters.

Our main result is the following {\em $l^p$ Decoupling Theorem}.
\begin{theorem}
\label{t1}

Let $n\ge 2$.  If $\supp(\hat{f})\subset \A_\delta$
then for $p\ge\frac{2(n+1)}{n-1}$ and $\epsilon>0$
\begin{equation}
\label{e6}
\|f\|_p\lesssim_\epsilon \delta^{\frac{n}p-\frac{n-1}{2}-\epsilon}(\sum_{\theta\in \P_\delta}\|f_\theta\|_p^p)^{1/p}.
\end{equation}
\end{theorem}
This is a close cousin of the following $l^2$ decoupling proved in \cite{BD3} in the case when $S$ has definite second fundamental form
\begin{equation}
\label{e41}
 \|f\|_p\lesssim_\epsilon \delta^{-\frac{n-1}4+\frac{n+1}{2p}-\epsilon}(\sum_{\theta\in \P_\delta}\|f_\theta\|_p^2)^{1/2},\;\;p\ge\frac{2(n+1)}{n-1}.
\end{equation}
We point out that \eqref{e6} is slightly weaker than \eqref{e41} as it follows from  \eqref{e41} via H\"older's inequality. In particular, Theorem \ref{t1} for $n=2$ is contained in \cite{BD3}. We mention that sharp $l^p$ decouplings were first considered by Wolff  in the case of the cone, see \cite{TWol}.

As briefly explained in \cite{BD3}, \eqref{e41} is false for the hyperbolic paraboloid due to the fact that it contains lines. On the other hand, apart from the dependence on $\epsilon$, inequality \eqref{e6} is sharp, as is \eqref{e41}. This can be easily seen by considering the case of the sphere $S=S^{n-1}$ and $f$ with $\widehat{f}=1_{\A_\delta}$.
 The main new difficulty in proving \eqref{e6} as compared to \eqref{e41} is the fact that intersections of hyperbolic paraboloids with hyperplanes do not always have nonzero Gaussian curvature. It will  however be crucial to our argument the fact that at most one of the principal curvatures of these sections can be small.
\bigskip

As an application we consider curves $\Phi:[0,1]\to\R^n$, $\Phi(t)=(\phi_1(t),\ldots,\phi_{n}(t))$ with $\phi_i\in C^n([0,1])$ and such that the Wronskian
$$W(\phi_1',\ldots,\phi_{n}')(t)$$
is nonzero on $[0,1]$.
Note in particular that such a  curve does not lie in a hyperplane.

Abusing earlier notation, let $\A_\delta$ be the $\delta$ neighborhood of $\Phi$ and let $\P_\delta$ be the cover of  $\A_\delta$ with $\delta$ neighborhoods $\theta$ of the restrictions of $\Phi$ to dyadic intervals of length $\delta^{1/n}$.
 We will as before denote by $f_\theta$ the Fourier restriction of $f$ to $\theta$.
\begin{theorem}
\label{t2}
For each such curve $\Phi$ and each $f:\R^n\to\C$ with Fourier support in $\A_\delta$ we have
$$\|f\|_{L^{2(n+1)}(\R^n)}\lesssim_\epsilon\delta^{-\frac1{2(n+1)}-\epsilon}(\sum_{\theta\in \P_\delta}\|f_\theta\|_{L^{2(n+1)}(\R^n)}^{2(n+1)})^{\frac1{2(n+1)}},$$
for each $\epsilon>0$.
\end{theorem}

By using a limiting procedure (see for example the discussion in \cite{BD3} and \eqref{e11} below) one can replace $f$ with a sum of Dirac deltas and derive the following corollary.
\begin{corollary}
\label{corrrrr771}
Fix $\Phi$ as above and let $p\le 2(n+1)$. Then for each $\delta$-separated set $\Lambda$ which consists of $\sim \delta^{-1}$ points on the curve $\Phi$ we have
\begin{equation}
\label{ropig90580vyunmv45r0-c489rt870-910-}
(\frac{1}{|B_R|}\int_{B_R}|\sum_{\xi\in\Lambda}e(\xi\cdot x)|^pdx)^{1/p}\lesssim_\epsilon \delta^{-\epsilon}|\Lambda|^{1/2},
\end{equation}
for each $\epsilon$ and each ball $B_R\subset \R^n$ of radius $R\gtrsim \delta^{-n}$.
\end{corollary}

Given an integer $k\ge 2$ we define the  $k$-energy of $\Lambda$ as
$$\B_k(\Lambda)=|\{(\lambda_1,\ldots,\lambda_{2k})\in \Lambda^{2k}:\;\lambda_1+\ldots+\lambda_k=\lambda_{k+1}+\ldots+\lambda_{2k}\}|.$$
By letting $R\to\infty$ in \eqref{ropig90580vyunmv45r0-c489rt870-910-} with $p=2(n+1)$ we immediately get that $$\B_{n+1}(\Lambda)\lesssim_\epsilon \Lambda^{n+1+\epsilon}$$  for each $\Lambda$ as in Corollary \ref{corrrrr771}. In particular, by applying this to the curve
$$\Phi(t)=(t,t^2,\ldots,t^n)$$
we recover (again, apart from the $\epsilon$ loss) the result of Hua \cite{Hu}
$$\B_{n+1}(\{(l,l^2,\ldots,l^n):\;l=1,2,\ldots,N\})\lesssim_\epsilon N^{n+1+\epsilon}.$$
Our method however shows that the integer case is not special, but is rather a particular case of a larger phenomenon. 

Further applications of variants of inequality \eqref{ropig90580vyunmv45r0-c489rt870-910-} to number theory are presented in \cite{Bo}.

\section{$l^p$ decouplings for hypersurfaces}
\bigskip

In this section we present the proof of Theorem \ref{t1}. We start by observing that the induction on scales argument from the last section in \cite{BD3} allows us to focus on the hypersurfaces $H^{n-1}_\upsilon$.

The proof of Theorem \ref{t1} for $H^{n-1}_\upsilon$  will be done in two separate stages. First, we develop the multilinear decoupling theory and show that it is essentially equivalent to its linear counterpart.

Let $g:H^{n-1}_\upsilon\to\C$. For a cap $\tau$ on $H^{n-1}_\upsilon$ we let $g_\tau=g1_\tau$ be the (spatial) restriction of $g$ to $\tau$.
We denote by $\pi:H^{n-1}_\upsilon\to [-1/2,1/2]^{n-1}$ the projection map and by $d\sigma$ the natural surface measure on $H^{n-1}_\upsilon$.
\begin{definition}
We say that the caps $\tau_1,\ldots,\tau_n$ on $H^{n-1}_\upsilon$ are $\nu$-transverse if the volume of the parallelepiped spanned by any unit normals $n_i$ at $\tau_i$ is greater than $\nu$.
\end{definition}

In the following, the norm $\|f\|_{L^p(w_{B_R})}$ will refer to the weighted $L^p$ integral $$(\int_{\R^n}|f(x)|^pw_{B_R}(x)dx)^{1/p}$$
for some weight $w_{B_R}$ which is Fourier supported in $B(0,\frac{1}{R})$ and satisfies
\begin{equation}
\label{wnret78u-0943mt7-w,-,1ir8gnrnfyqerbtf879wumweryermf,}
1_{B_R}(x)\lesssim w_{B_R}(x)\le(1+\frac{|x-c(B_R)|}{R})^{-10n}.
\end{equation}

We denote by $C_{p,n,\upsilon}(\delta,\nu)$ the smallest constant such that
$$\|(\prod_{i=1}^n|\widehat{g_{\tau_i}d\sigma}|)^{1/n}\|_{L^p(B_{\delta^{-1}})}\le C_{p,n,\upsilon}(\delta,\nu)\left[\prod_{i=1}^n(\sum_{\theta:\;\delta^{1/2}-\text{cap}\atop{\theta\subset\tau_i}}\|\widehat{g_{\theta}d\sigma}\|_{L^p(w_{B_{\delta^{-1}}})}^p)^{1/p}\right]^{1/n},$$
for each $\nu$-transverse caps $\tau_i\subset H^{n-1}_\upsilon$, each $\delta^{-1}$ ball $B_{\delta^{-1}}$ and each $g:H^{n-1}_\upsilon\to\C$.

Let also $K_{p,n,\upsilon}(\delta)$ be the smallest constant such that
$$\|\widehat{gd\sigma}\|_{L^p(B_{\delta^{-1}})}\le K_{p,n,\upsilon}(\delta)(\sum_{\theta:\delta^{1/2}-\text{cap}}\|\widehat{g_{\theta}d\sigma}\|_{L^p(w_{B_{\delta^{-1}}})}^p)^{1/p},$$
for each $g:H^{n-1}_\upsilon\to\C$ and each $\delta^{-1}$ ball $B_{\delta^{-1}}$.

\bigskip

H\"older's inequality gives
$$C_{p,n,\upsilon}(\delta,\nu)\le K_{p,n,\upsilon}(\delta).$$
We will show that the reverse inequality essentially holds true. Recall that the ultimate goal is to prove that for $2\le p\le \frac{2(n+1)}{n-1}$
\begin{equation}
\label{e15}
K_{p,n,\upsilon}(\delta)\lesssim_\epsilon \delta^{-\frac{n-1}{2}(\frac12-\frac1p)-\epsilon}.
\end{equation}

\begin{theorem}
\label{t4}
Fix $n\ge 3$, $\upsilon\in\{-1,1\}^{n-1}$ and let $p\ge 2$. Assume one of the following holds

(i) $n=3$

(ii) $n\ge 4$ and
\begin{equation}
\label{e42}
K_{p,n-2,\upsilon'}(\delta')\lesssim_\epsilon {\delta'}^{-\frac{n-3}{2}(\frac12-\frac1p)-\epsilon}
\end{equation}
 for each $\delta'>0$, $\upsilon'\in\{-1,1\}^{n-3}$ and  each $\epsilon>0$.

Then for each $0<\nu\le 1$  there is $\epsilon(\nu)$ with $\lim_{\nu\to 0}\epsilon (\nu)=0$ and $C_\nu$ such that
$$K_{p,n,\upsilon}(R^{-1})\le C_\nu R^{\epsilon(\nu)}C_{p,n,\upsilon}(R^{-1},\nu)$$
for each $R>1$.
\end{theorem}

We first prove a lemma for paraboloids that are allowed to have one small (possibly zero) principal curvature. The motivation behind this consideration will be explained in the end of the proof of Proposition \ref{hcnyf7yt75ycn8u32r8907n580-9=--qc mvntvu5n8t}.
\begin{lemma}
\label{l2}
Let $n\ge 3$. Fix $\upsilon_1,\ldots,\upsilon_{n-2}\in\{-1,1\}$ and let $|a|\lesssim 1$ be arbitrary, possibly zero. Let $\P_\delta$ be a partition of the neighborhood $\A_\delta$ associated with the hypersurface $H^{n-1}_{(\upsilon_1,\ldots,\upsilon_{n-2},a)}$.

If $\supp(\hat{f})\subset \A_\delta$
then for $p\ge 2$ we have, uniformly over the parameter $|a|\lesssim 1$
$$
\|f\|_p\lesssim \delta^{-\frac12+\frac1p}K_{p,n-1,(\upsilon_1,\ldots,\upsilon_{n-2})}(\delta)(\sum_{\theta\in \P_\delta}\|f_\theta\|_p^p)^{1/p}.
$$
\end{lemma}
Before we embark in the proof of the lemma, we give some heuristics on numerology. A simple $L^p$ orthogonality principle asserts that given any pairwise disjoint subsets $S_1,\ldots,S_{M}$ in $\R^n$ we have
\begin{equation}
\label{e12}
\|f\|_p\le M^{1-\frac2p}(\sum_{i=1}^{M} \|f_{S_i}\|_p^p)^{1/p}
\end{equation}
for each $2\le p\le \infty$ and each  $f$ Fourier supported in the union of the $S_i$. We may refer to this as being {\em trivial $l^p$ decoupling}.
 If the sets $S_i$ are arbitrary, the universal exponent $1-\frac2p$ of $M$ is sharp. To see, it suffices to consider the case when $S_i$ are equidistant unit balls with collinear centers.  However, this exponent becomes smaller when geometry is favorable. For example, the $L^p$ decoupling inequality \eqref{e15} corresponds to $M\sim \delta^{-\frac{n-1}{2}}$, and the exponent there is $\frac12-\frac1p$,  half of the universal one. The absence of curvature is an enemy, and one expects a penalty of $\delta^{\frac1{2p}-\frac14}$ for each zero principal curvature.  For example, when $a$ is small (possibly zero),  $H^{n-1}_{(\upsilon_1,\ldots,\upsilon_{n-2},a)}$ has a decoupling constant  $\delta^{\frac1{2p}-\frac14}$ larger than in the case $a\sim 1$.
\bigskip

\begin{proof}
The proof is rather standard, we sketch it briefly. The case $n=3$ is entirely typical, we prefer it only to simplify the notation. We can of course also assume $\upsilon_1=1$. By first performing the trivial decoupling \eqref{e12} in the direction of $e_2$ (which corresponds to the entry $a$), it suffices to prove that for each $\alpha\in [-\frac12,\frac12]$ we have
\begin{equation}
\label{e13}
\|\sum_{\theta\in \P_{\delta,\alpha}}f_\theta\|_p\lesssim K_{p,2,1}(\delta)(\sum_{\theta\in \P_{\delta,\alpha}}\|f_\theta\|_p^p)^{1/p},
\end{equation}
where $\P_{\delta,\alpha}$ consists of those $\theta\in \P_\delta$ that intersect the parabola
$$\{(\xi_1,\alpha,\xi_3):\xi_3=\xi_1^2+a\alpha^2\}.$$

We next show how a standard parabolic change of coordinates will allow us to assume $\alpha=0$.
It is easy to see (the reader is again referred to \cite{BD3}) that \eqref{e13} is equivalent with
$$\int_{B_{\delta^{-1}}}|\int_{[-1/2,1/2]\times [\alpha,\alpha+\delta^{1/2}]}f(\xi_1,\xi_2)e(x_1\xi_1+x_2\xi_2+x_3(\xi_1^2+a\xi_2^2))d\xi_1d\xi_2|^pdx_1dx_2dx_3\lesssim $$$$ K_{p,2,1}(\delta)^p \sum_{I:\delta^{1/2}-\text{interval}}\int_{B_{\delta^{-1}}}|\int_{I\times [\alpha,\alpha+\delta^{1/2}]}f(\xi_1,\xi_2)e(x_1\xi_1+x_2\xi_2+x_3(\xi_1^2+a\xi_2^2))d\xi_1d\xi_2|^pdx_1dx_2dx_3$$
for each $B_{\delta^{-1}}$. It is now rather immediate that we can assume $\alpha=0$, since the image of $B_{\delta^{-1}}$ under the transformation
$$(x_1,x_2,x_3)\mapsto (x_1,x_2+2\alpha a x_3,x_3)$$
is roughly $B_{\delta^{-1}}$.

Note however that all $\theta\in\P_{\delta,0}$ lie in the $\delta$ neighborhood of the cylinder
$$\{(\xi_1,\xi_2,\xi_3):\xi_3=\xi_1^2\},$$
and \eqref{e13} follows immediately from Fubini.
\end{proof}
\bigskip

A simple induction on scales similar to the one in Section 7 in \cite{BD3} allows us to extend the previous lemma to arbitrary hypersurfaces with (at least) $n-1$ principal  curvatures bounded away from zero.

\begin{lemma}
\label{l7}
 Let $n\ge 3$ and $p\ge 2$. Let $S$ be a $C^2$ compact hypersurface in $\R^n$ which at any given point has at least $n-1$ principal curvatures $\sim 1$ with the remaining one $\lesssim 1$. Let as usual $\P_\delta$ be a partition of the neighborhood $\A_\delta$ associated $S$.
Assume that for each $\delta'>0$
$$\max_{\upsilon\in\{-1,1\}^{n-2}}K_{p,n-1,\upsilon}(\delta')\lesssim_{\epsilon}{\delta'}^{-\frac{n-2}{2}(\frac12-\frac1p)-\epsilon}.$$

Then for each $\delta$ and each $\supp(\hat{f})\subset \A_\delta$
we have
$$
\|f\|_p\lesssim_{\epsilon}{\delta}^{-\frac{n}{2}(\frac12-\frac1p)-\epsilon}(\sum_{\theta\in \P_\delta}\|f_\theta\|_p^p)^{1/p}.
$$
\end{lemma}
\begin{proof}
For $\delta<1$, let as before $K_{p,n,S}(\delta)$ be the smallest constant such that for each   $f$ with Fourier support in $\A_\delta$ we have
$$\|f\|_p\le K_{p,n,S}(\delta)(\sum_{\theta\in \P_\delta}\|f_\theta\|_p^2)^{1/2}.$$
First note that for each such $f$
$$\|f\|_p\le K_{p,n,S}(\delta^{\frac23})(\sum_{\tau\in \P_{\delta^{\frac23}}}\|f_\tau\|_p^2)^{1/2}.$$
Second, our assumption on the principal curvatures of $S$ combined with Taylor's formula shows that on each $\tau\in \P_{\delta^{\frac23}}$, $S$ is within $\delta$ from a paraboloid $H^{n-1}_{\upsilon}$ with at least $n-2$ of the entries of $\upsilon$ of order 1. By invoking Lemma \ref{l2} for this paraboloid (via a simple rescaling), combined with  parabolic rescaling we get
$$\|f_\tau\|_p\lesssim (\delta^{1/3})^{\frac1p-\frac12}\max_{\upsilon\in\{-1,1\}^{n-2}}K_{p,n-1,\upsilon}(\delta^{1/3})(\sum_{\theta\in \P_\delta:\theta\subset\tau}\|f_\theta\|_p^2)^{1/2}.$$
For each $\epsilon>0$, we conclude the existence of $C_\epsilon$ such that for each $\delta<1$
$$K_{p,n,S}(\delta)\le C_\epsilon [\delta^{-\frac{n}{2}(\frac12-\frac1p)-\epsilon}]^{1/3}K_{p,n,S}(\delta^{\frac23}).$$
By iteration this immediately leads to the desired conclusion.
\end{proof}

We next present a lemma that will play a key role in the proof of Proposition \ref{hcnyf7yt75ycn8u32r8907n580-9=--qc mvntvu5n8t} below.

\begin{lemma}
\label{l:121}
Let $A$ be an invertible symmetric $n\times n$ matrix and let $S$ be an $m$ dimensional affine subspace of $\R^n$. There exists $\delta=\delta(A)$ such that if the $m$ dimensional quadratic hypersurface
$$x_{m+1}=\langle Ax,x\rangle,\;x\in S$$
has $l$ principal curvatures in the interval $[-\delta,\delta]$ then
$$l\le n-m.$$
\end{lemma}
\begin{proof}
We may assume $S$ contains the origin. Choose $\delta$ small enough so that the hypothesis forces the existence of an $l$ dimensional subspace $S_1$ of $S$ such that
\begin{equation}
\label{3dvjher9r576tiocuxmgyutxmcjruit78}
\|P_{S}Ax\|\le \frac12\|A^{-1}\|\|x\|
\end{equation}
for each $x\in S_1$. Here $P_S$ is the orthogonal projection onto $S$. We claim that $S_1\cap A^{-1}S=\{0\}$, which will easily imply the desired conclusion. Indeed, otherwise there is $x\in S_1$ with $\|x\|=1$ and $Ax\in S$, and \eqref{3dvjher9r576tiocuxmgyutxmcjruit78} forces the contradiction.
\end{proof}
\bigskip

Here is the basic step in the Bourgain-Guth-type induction on scales that relates the linear and the multilinear decoupling.

\begin{proposition}
\label{hcnyf7yt75ycn8u32r8907n580-9=--qc mvntvu5n8t}
Fix $n\ge 3$, $\upsilon\in\{-1,1\}^{n-1}$ and let $p\ge 2$. Assume one of the following holds

(i) $n=3$

(ii) $n\ge 4$ and $K_{p,n-2,\upsilon'}(\delta')\lesssim_\epsilon {\delta'}^{-\frac{n-3}{2}(\frac12-\frac1p)-\epsilon}$ for each $\delta'>0$, $\upsilon'\in\{-1,1\}^{n-3}$ and  each $\epsilon>0$.

 Then for each $\epsilon$ there exist constants $C_\epsilon$, $C_n$ such that for each $R>1$ and $K\ge 1$
$$\|\widehat{gd\sigma}\|_{L^p(w_{B_R})}\le C_\epsilon[(\sum_{\alpha\subset H^{n-1}_\upsilon\atop{\alpha:\frac{1}{K}-\text{ cap}}}\|\widehat{g_{\alpha}d\sigma}\|_{L^p(w_{B_R)}}^p)^{1/p}+K^{\frac{n-1}{2}(\frac12-\frac1p)+\epsilon}(\sum_{\beta\subset H^{n-1}_\upsilon\atop{\beta:\frac{1}{K^{1/2}}-\text{ cap}}}\|\widehat{g_{\beta}d\sigma}\|_{L^p(w_{B_R})}^p)^{1/p}]+$$$$+K^{C_n}C_{p,n,\upsilon}(R^{-1},K^{-n})(\sum_{\Delta\subset H^{n-1}_\upsilon\atop{\Delta:\frac1{R^{1/2}}-\text{ cap}}}\|\widehat{g_{\Delta}d\sigma}\|_{L^p(w_{B_R})}^p)^{1/p}$$
\end{proposition}
\begin{proof}

We first prove the case $n=3$ and then indicate the modifications needed for $n\ge 4$.

It is rather immediate that if  $Q_1,Q_2,Q_3\subset [-1/2,1/2]^2$, the volume of the parallelepiped spanned by the  unit normals to $H^2_\upsilon$ at $\pi^{-1}(Q_i)$ is comparable to the area of the triangle $\Delta Q_1Q_2Q_3$.

As in \cite{BG}, we may think of $|\widehat{g_{\alpha}d\sigma}|$ as being essentially constant on each ball $B_K$. Denote by $c_\alpha(B_K)$ this value and  let $\alpha^*$ be the cap that maximizes it.

The starting point in the argument is the observation in \cite{BG} that for each $B_K$ there exists a line $L=L(B_K)$ in the $(\xi_1,\xi_2)$ plane such that if
$$S_L=\{(\xi_1,\xi_2): \dist((\xi_1,\xi_2),L)\le \frac{C}{K}\}$$
then for $x\in B_K$
$$ |\widehat{gd\sigma}(x)|\le $$
\begin{equation}
\label{term1.1}C\max_{\alpha}|\widehat{g_{\alpha}d\sigma}(x)|+
\end{equation}\begin{equation}
\label{term1.2}
K^{4}\max_{\alpha_1,\alpha_2,\alpha_3\atop{K^{-2}-\text{transverse}}}(\prod_{i=1}^3|\widehat{g_{\alpha_i}d\sigma}(x)|)^{1/3}+
\end{equation}\begin{equation}
\label{term1.3}
|\sum_{\alpha\subset \pi^{-1}(S_L)\cap H^2_\upsilon}\widehat{g_{\alpha}d\sigma}(x)|.
\end{equation}
To see this, we distinguish three scenarios. First, if $c_\alpha(B_K)\le K^{-2}c_{\alpha^*}(B_K)$ for each $\alpha$ with  $\dist(\pi(\alpha),\pi(\alpha^*))\ge \frac{10}{K}$, then \eqref{term1.1} suffices, as
$$ |\widehat{gd\sigma}(x)|\le \sum_{\alpha}c_\alpha(B_K).$$
Otherwise, there is  $\alpha^{**}$ with $\dist(\pi(\alpha^{**}),\pi(\alpha^*))\ge \frac{10}{K}$ and
$c_{\alpha^{**}}(B_K)\ge K^{-2}c_{\alpha^*}(B_K)$.
The line $L$ is determined by $\alpha^*,\alpha^{**}$.

Second, if there is $\alpha^{***}$ such that  $\pi(\alpha^{***})$ intersects the complement of $S_L$ and $c_{\alpha^{***}}(B_K)\ge K^{-2}c_{\alpha^*}(B_K)$ then  \eqref{term1.2} suffices. Indeed, note that $\alpha^{*},\alpha^{**}$, $\alpha^{***}$ are $K^{-2}$ transverse.

Otherwise  the sum of \eqref{term1.1} and \eqref{term1.3} will suffice.

The only nontrivial case to address is the one corresponding to this latter scenario. Cover $\pi^{-1}(S_L)\cap H^2_\upsilon$ by pairwise disjoint strips $U$ of length $\sim \frac{1}{K^{1/2}}$. An application of the trivial $l^p$ decoupling \eqref{e12} shows that
$$\|\sum_{\alpha:\pi(\alpha)\subset S_L}\widehat{g_{\alpha}d\sigma}\|_{L^p(B_K)}\lesssim K^{\frac12-\frac1p}(\sum_{U}\|\widehat{g_{U}d\sigma}\|_{L^p(w_{B_K})}^p)^{1/p}.$$
This is the best we can say in general. Indeed, in the case of the hyperbolic paraboloid $\upsilon=(1,-1)$, if the line $L$ happens to be $\xi_2=\pm\xi_1$ then $\pi^{-1}(L)$ is itself a line. The absence of curvature prevents any non-trivial estimate to hold.

 Note however that since we are dealing with the third scenario,
$$(\sum_{U}\|\widehat{g_{U}d\sigma}\|_{L^p(B_K)}^p)^{1/p}\lesssim (\sum_{\beta:\frac{1}{K^{1/2}}-\text{ cap}:\atop{\pi(\beta)\subset S_L}}\|\widehat{g_{\beta}d\sigma}\|_{L^p(w_{B_K})}^p)^{1/p}+\|\widehat{g_{\alpha^*}d\sigma}\|_{L^p(w_{B_K})}.$$
We conclude that in either case
$$\|\widehat{gd\sigma}\|_{L^p(B_K)}\lesssim [(\sum_{\alpha\subset H^2_\upsilon\atop{\alpha:\frac{1}{K}\text{ cap}}}\|\widehat{g_{\alpha}d\sigma}\|_{L^p(w_{B_K})}^p)^{1/p}+K(\sum_{\beta\subset H^2_\upsilon\atop{\beta:\frac{1}{K^{1/2}}\text{ cap}}}\|\widehat{g_{\beta}d\sigma}\|_{L^p(w_{B_K})}^p)^{1/p}]+$$$$+K^{10}C_{p,3,\upsilon}(R^{-1},K^{-2})(\sum_{\Delta\subset H^2_\upsilon\atop{\Delta:\frac1{R^{1/2}}\text{ cap}}}\|\widehat{g_{\Delta}d\sigma}\|_{L^p(w_{B_K})}^p)^{1/p}.$$Finally,  raise to the $p^{th}$ power and sum over $B_K\subset B_R$. Also, the norm $\|\widehat{gd\sigma}\|_{L^{p}(B_R)}$ can be replaced by the weighted norm $\|\widehat{gd\sigma}\|_{L^{p}(w_{B_R})}$ via the standard localization argument.

One may repeat this argument in the case $n\ge 4$ as follows. For each $B_K$ there exists a hyperplane $\E=\E(B_K)$ in the $(\xi_1,\ldots,\xi_{n-2})$ space such that
for $x\in B_K$
$$ |\widehat{gd\sigma}(x)|\le $$
$$C\max_{\alpha}|\widehat{g_{\alpha}d\sigma}(x)|+$$
$$K^{C_n}\max_{\alpha_1,\ldots,\alpha_n\atop{K^{-n}-\text{transverse}}}(\prod_{i=1}^n|\widehat{g_{\alpha_i}d\sigma}(x)|)^{1/n}+$$
$$|\sum_{\alpha\subset \pi^{-1}(S_\E)\cap H^{n-1}_\upsilon}\widehat{g_{\alpha}d\sigma}(x)|.$$
Here
$$S_\E=\{(\xi_1,\ldots,\xi_{n-1}): \dist((\xi_1,\ldots,\xi_{n-1}),\E)\lesssim \frac{1}{K}\}$$
We only need to explain how to accommodate the previous argument to control the last term. Cover $\pi^{-1}(S_\E)\cap H^{n-1}_\upsilon$ by pairwise disjoint strips $U$ of dimension $\sim \frac{1}{K^{1/2}}\times\ldots\times \frac1{K^{1/2}}\times \frac1K$. These strips are inside the $\frac1K$ neighborhood of the $n-2$ dimensional manifold $$S_{\E,\upsilon}=\{(\xi_1,\ldots,\xi_n)\in H^{n-1}_\upsilon:(\xi_1,\ldots,\xi_{n-1})\in \E\},$$
and correspond to a tiling of this manifold by $\frac1{K^{1/2}}$- caps. The important new observation is that $S_{\E,\upsilon}$, regarded as a hypersurface in the hyperplane
$$\{(\xi_1,\ldots,\xi_n):(\xi_1,\ldots,\xi_{n-1})\in \E\}$$
has at least $n-2$ of its $n-1$ principal curvatures bounded away from zero, at any given point. This is of course a consequence of Lemma \ref{l:121}. The case $n=3$ discussed earlier shows that one (in this case the only) principal curvature may indeed happen to be zero. More generally, consider any hyperbolic paraboloid $H^{n-1}_{\upsilon}$. Fix any $A_1,\ldots,A_{n-1}$ such that
$$\sum_{i=1}^{n-1}\upsilon_iA_i^2=0.$$Let $\E$ be the hyperplane
$$\sum_{i=1}^{n-1}\upsilon_iA_i\xi_i=0$$

It is easy to check that for each point $(\xi_1^*,\ldots,\xi_n^*)$ in the corresponding manifold $S_{\E,\upsilon}$, the (appropriate part of the) line
$$\frac{\xi_1-\xi_1^*}{A_1}=\ldots=\frac{\xi_{n-1}-\xi_{n-1}^*}{A_{n-1}}=\frac{\xi_n-\xi_n^*}{0}$$
is inside $S_{\E,\upsilon}$. In other words $S_{\E,\upsilon}$ is a cylinder, and one of its principal curvatures will be zero.

Using our hypothesis and  Lemma \ref{l7} we can write
$$\|\sum_{\alpha\subset \pi^{-1}(S_\E)\cap H^{n-1}_\upsilon}\widehat{g_{\alpha}d\sigma}\|_{L^p(B_K)}\lesssim_\epsilon K^{\frac{n-1}{2}(\frac12-\frac1p)+\epsilon}(\sum_{U}\|\widehat{g_{U}d\sigma}\|_{L^p(w_{B_K})}^p)^{1/p}.$$
As before, this can be further bounded by

$$\lesssim K^{\frac{n-1}{2}(\frac12-\frac1p)+\epsilon}[(\sum_{\beta:\frac1{K^{1/2}}-cap}\|\widehat{g_{\beta}d\sigma}\|_{L^p(w_{B_K})}^p)^{1/p}+(\sum_{\alpha:\frac1{K}-cap}\|\widehat{g_{\alpha}d\sigma}\|_{L^p(w_{B_K})}^p)^{1/p}].$$
The argument is now complete.
\end{proof}
\bigskip

Simple parabolic rescaling leads to the following more general version. The interested reader should consult the proof of the analogous result in \cite{BD3} for details.
\begin{proposition}
\label{p9}
Fix $n\ge 3$, $\upsilon\in\{-1,1\}^{n-1}$ and let $p\ge 2$. Assume one of the following holds

(i) $n=3$

(ii) $n\ge 4$ and $K_{p,n-2,\upsilon'}(\delta')\lesssim_\epsilon {\delta'}^{-\frac{n-3}{2}(\frac12-\frac1p)-\epsilon}$ for each $\delta'>0$, $\upsilon'\in\{-1,1\}^{n-3}$ and  each $\epsilon>0$.

 Then for each $\epsilon$ there exist constants $C_\epsilon$, $C_n$ such that for each $R>1$ and $K\ge 1$ and for each $\delta$-cap $\tau$ on $H^{n-1}_\upsilon$ we have
$$\|\widehat{g_\tau d\sigma}\|_{L^p(w_{B_R})}\le C_\epsilon[(\sum_{\alpha\subset \tau\atop{\alpha:\frac{\delta}{K}-\text{ cap}}}\|\widehat{g_{\alpha}d\sigma}\|_{L^p(w_{B_R)}}^p)^{1/p}+K^{\frac{n-1}{2}(\frac12-\frac1p)+\epsilon}(\sum_{\beta\subset \tau\atop{\beta:\frac{\delta}{K^{1/2}}-\text{ cap}}}\|\widehat{g_{\beta}d\sigma}\|_{L^p(w_{B_R})}^p)^{1/p}]+$$$$+K^{C_n}C_{p,n,\upsilon}((R\delta^2)^{-1},K^{-n})(\sum_{\Delta\subset \tau\atop{\Delta:\frac1{R^{1/2}}-\text{ cap}}}\|\widehat{g_{\Delta}d\sigma}\|_{L^p(w_{B_R})}^p)^{1/p}$$
\end{proposition}

We are now ready to prove Theorem \ref{t4}. Let $K=\nu^{-1/n}$. We can certainly assume that the quantity $\delta^{\frac{n-1}{2}(\frac12-\frac1p)}C_{p,n,\upsilon}(\delta,\nu)$ is an essentially decreasing function of $\delta>0$, otherwise the results we aim to prove become trivial. This can be worked out rigorously as in \cite{BD3} using parabolic rescaling, we leave the details to the reader.
In particular, we may assume that
\begin{equation}
\label{e16}
\delta^{\frac{n-1}{2}(\frac12-\frac1p)}C_{p,n,\upsilon}(\delta,\nu)\lesssim_\epsilon R^{-\frac{n-1}{2}(\frac12-\frac1p)+\epsilon}C_{p,n,\upsilon}(R^{-1},\nu)
\end{equation}
for each $\delta>R^{-1}$ and $\epsilon>0$ .

Iterate Proposition \ref{p9} starting with caps of scale $ 1$ until all resulting caps have scale  $R^{-1/2}$.
Each iteration  lowers the scale of the caps from $\delta$ to at least $\frac{\delta}{K^{1/2}}$. When iteration is over, we end up with a sum of terms of the form
$$T_\Gamma=\Gamma K^{C_n}(\sum_{\Delta:\frac1{R^{1/2}}-\text{ cap}}\|\widehat{g_{\Delta}d\sigma}\|_{L^p(w_{B_R})}^p)^{1/p},$$
with various coefficients $\Gamma$.
Each such term  arises via $\le\log_K R$ iterations. Also, a crude estimate shows that we end up with at most $3^{\log_K R}=R^{O(\log_{\nu^{-1}}3)}$ such terms.

It remains to get a uniform upper bound on $\Gamma$. Tracing back the iteration history of $T_\Gamma$, assume it went through $m_1$ steps where scale was lowered by $K$ and $m_2$ steps where scale was lowered by $K^{1/2}$. Then obviously, for each $\epsilon$
$$\Gamma\le (C_{\epsilon})^{m_1+m_2}K^{[\frac{n-1}{2}(\frac12-\frac1p)+\epsilon]m_2}C_{p,n,\upsilon}((RK^{-m_2-2m_1})^{-1},\nu).$$
Using  the bound $m_1+m_2\le \log_K R$ this is further bounded by
$$R^{\log_{\nu^{-1}}C_\epsilon}K^{\epsilon\log_KR}C_{p,n,\upsilon}((RK^{-m_2-2m_1})^{-1},\nu)(RK^{-m_2-2m_1})^{-\frac{n-1}{2}(\frac12-\frac1p)}R^{\frac{n-1}{2}(\frac12-\frac1p)}$$

Finally, \eqref{e16} shows that
$$\Gamma\lesssim_\epsilon R^{\epsilon+\log_{\nu^{-1}}C_\epsilon}C_{p,n,\upsilon}(R^{-1},\nu).$$
The proof is now complete, by carefully letting $\epsilon$ approach zero at slower rate than $\nu$.
\bigskip

We now enter the second and final stage of the argument for Theorem \ref{t1}. For the remainder of the section we fix $\upsilon\in\{-1,1\}^{n-1}$ and $p>\frac{2(n+1)}{n-1}$.

 As in \cite{BD3}, let
$$\gamma=\liminf_{\delta\to 0}\frac{\log K_{p,n,\upsilon}(\delta)}{\log(\delta^{-1})}.$$
It follows that for each $\epsilon$
$$\delta^{-\gamma}\lesssim K_{p,n,\upsilon}(\delta)\lesssim_\epsilon \delta^{-\gamma-\epsilon}.$$
Write $\gamma=\frac{n-1}{4}-\frac{n+1}{2p}+\alpha$. For the rest of the argument we will assume that $\alpha>0$,  and will show how to reach a contradiction.

Define $$\xi=\frac{2}{(p-2)(n-1)}$$ $$\eta=\frac{n(np-2n-p-2)}{2p(n-1)^2(p-2)}.$$Since $p>\frac{2(n+1)}{n-1}$ we have that $\xi<\frac12$. A simple computation reveals that the assumption $\alpha>0$ is equivalent with
$$\gamma\frac{1-\xi}{1-2\xi}>\frac{n-1}{4}-\frac{n^2+n}{2p(n-1)}+\frac{2\eta}{1-2\xi}.$$
It follows that we can choose $s_0\in\N$  large enough and  $\epsilon_0$ small enough so that
$$\gamma(\frac{1-\xi}{1-2\xi}-\frac{\xi(2\xi)^{s_0}}{1-2\xi})>
$$
\begin{equation}
\label{mamatata}
\frac{n-1}{4}-\frac{n^2+n}{2p(n-1)}+2^{s_0}\epsilon_0+\frac{2\eta}{1-2\xi}(1-(2\xi)^{s_0})+\frac{n}{(n-1)p}(2\xi)^{s_0}.
\end{equation}

Choose $\nu>0$ small enough such that $\epsilon_0>\epsilon(\nu)$, with $\epsilon(\nu)$ as in  Theorem \ref{t4}.
Note that $s_0$, $\epsilon_0$ and $\nu$ depend only on the fixed parameters $p,n,\alpha$. As a result, we follow our convention and do not record the dependence on them when using the symbol $\lesssim$.

Throughout the rest of the section $\nu$, $s_0$ and $\epsilon_0$ will always refer to these values. To simplify notation we let $K(\delta)=\delta^{\frac{n-1}{2}(\frac12-\frac1p)}K_{p,n,\upsilon}(\delta)$ and $C(\delta)=\delta^{\frac{n-1}{2}(\frac12-\frac1p)}C_{p,n,\upsilon}(\delta,\nu)$.
Introduce the following semi-norms
$$\|f\|_{p,\delta,B}=(\sum_{\theta\in \P_\delta}\|f_\theta\|_{L^p(w_B)}^2)^{1/2},$$
$$|||f|||_{p,\delta,B}=\delta^{-\frac{n-1}{2}(\frac12-\frac1p)}(\sum_{\theta\in \P_\delta}\|f_\theta\|_{L^p(w_B)}^p)^{1/p}$$
and note that
\begin{equation}
\label{e22}
\|f\|_{p,\delta,B}\le |||f|||_{p,\delta,B}
\end{equation}

For a fixed $\theta$ consider the inequality
\begin{equation}
\label{inv999}
\|(\prod_{i=1}^n|\widehat{g_id\sigma}|)^{1/n}\|_{L^p(B_N)}
\lesssim_{\epsilon,\theta} A_\theta(N)N^{\epsilon}X(B_N)^{1-\theta}Y(B_N)^\theta,
\end{equation}
for arbitrary $\epsilon>0$, $N$, $g_i$ and $B_N$ as before.
Here
$$X(B_N)=(\prod_{i=1}^n|||\widehat{g_id\sigma}|||
_{p,\delta,B_N})^{\frac1{n}},$$
$$Y(B_N)=(\prod_{i=1}^n|||\widehat{g_id\sigma}|||
_{\frac{p(n-1)}{n},\delta,B_N})^{\frac1{n}}.$$

The following holds.
\begin{proposition}
\label{propinv5}
(a) Inequality \eqref{inv999} holds true for $\theta=1$ with $A_1(N)=N^{\frac{n-1}{4}-\frac{n^2+n}{2p(n-1)}}$.

(b) Moreover, if we assume \eqref{inv999} for some $\theta\in (0,1]$, then we also have \eqref{inv999} for $\frac{2\theta}{(p-2)(n-1)}$ with
$$A_{\frac{2\theta}{(p-2)(n-1)}}(N)= A_\theta(N^{1/2})\delta^{-\frac{\gamma}{2}(1-\frac{2\theta}{(p-2)(n-1)})}N^{\frac{n(np-2n-p-2)}{2p(n-1)^2(p-2)}\theta}.$$
\end{proposition}

The proof follows line by line the proof of the analogous Proposition 6.3 in \cite{BD3}. More precisely, part (i) here follows right away from Proposition 6.3 (i) by simply invoking \eqref{e22}. Also, the only modification needed to prove part (ii) is to notice that the following consequence of H\"older's inequality used in \cite{BD3}
$$\|\widehat{g_id\sigma}\|_{\frac{p(n-1)}{n},\delta,B_N}\le \|\widehat{g_id\sigma}\|_{p,\delta,B_N}^{1-\frac{2}{(p-2)(n-1)}}\|\widehat{g_id\sigma}\|_{2,\delta,B_N}^{\frac{2}{(p-2)(n-1)}}$$
continues to hold if $\|\cdot\|$ is replaced with $|||\cdot|||$.

Proposition \ref{propinv5} implies that for each $s\ge 0$
$$A_{\xi^s}(N)= N^{\psi(\xi^s)}$$
with
\begin{equation}
\label{inv13}
\psi(\xi^{s+1})= \frac12\psi(\xi^s)+\frac\gamma2(1-\xi^{s+1})+\eta\xi^s.
\end{equation}
Recall that $\xi<\frac12$. Iterating \eqref{inv13}  gives
\begin{equation}
\label{inv15}
\psi(\xi^s)=\frac1{2^s}\psi(1)+\gamma(1-2^{-s})+2(\frac\eta\xi-\frac\gamma2)\frac{2^{-s}-\xi^{s}}{\xi^{-1}-2}
\end{equation}

Note that $Y(B_N)\lesssim X(B_N)N^{\frac{n}{(n-1)p}}$. As \eqref{inv999} holds for $\theta=\xi^s$ and arbitrary $\nu$-transverse caps $\tau_i$ we get
\begin{equation}
\label{7743048765898}
C(\delta)\lesssim_{\epsilon,s} \delta^{-\epsilon}A_{\xi^s}(N)N^{\frac{n\xi^s}{(n-1)p}}.
\end{equation}
\bigskip

To finish the proof of Theorem \ref{t1} for $H^{n-1}_\upsilon$, we will argue using induction on $n$ that $\alpha=0$. As observed earlier, the case $n=2$ is covered by the main theorem in \cite{BD3}. So the first case to consider is $n=3$.
Since \eqref{7743048765898} (with $s=s_0$) holds for arbitrarily small $\delta$ and $\epsilon$, using Theorem \ref{t4} we get

\begin{equation}
\label{inv14}
\gamma-\epsilon_0\le \psi(\xi^{s_0})+\frac{n\xi^{s_0}}{(n-1)p}.
\end{equation}
Combining \eqref{inv15} and \eqref{inv14}  we find
$$\gamma(\frac{1-\xi}{1-2\xi}-\frac{\xi(2\xi)^{s_0}}{1-2\xi})\le \psi(1)+2^{s_0}\epsilon_0+\frac{2\eta}{1-2\xi}(1-(2\xi)^{s_0})+\frac{n}{(n-1)p}(2\xi)^{s_0},$$
which  contradicts \eqref{mamatata}. Thus $\alpha=0$ and Theorem \ref{t1} is proved for $n=3$ and $p>4$.

Assume now that $n\ge 4$ and that Theorem \ref{t1} was proved in all dimensions  $d\le n-1$. To prove Theorem \ref{t1} in $\R^n$ for $p>\frac{2(n+1)}{n-1}$, it suffices to prove it for $\frac{2(n+1)}{n-1}<p<\frac{2(n-1)}{n-3}$. Note that in this range we have $p<\frac{2(d+1)}{d-1}$ for $d=n-2$, in particular \eqref{e42} holds. Thus Theorem \ref{t4} is applicable due to our induction hypothesis and we reach a contradiction as in the case $n=3$ discussed above.
\bigskip

It remains to see why Theorem \ref{t1} holds for the endpoint $p=p_n=\frac{2(n+1)}{n-1}$. Via a localization argument, $K_{p,n,\upsilon}(\delta)$ is comparable to the best constant $K_{p,n,\upsilon}^{*}(\delta)$ that makes the following inequality true for each $N$-ball $B_N$ and each $f$ Fourier supported in $\A_\delta$
\begin{equation}
\label{inv50}
\|f\|_{L^p({B_N})}\le K_{p,n,\upsilon}^{*}(\delta)(\sum_{\theta\in \P_\delta}\|f_\theta\|_{L^p(\R^n)}^p)^{1/p}.
\end{equation}
It suffices now to invoke Theorem \ref{t1} for $p>\frac{2(n+1)}{n-1}$ together with
$$\|f\|_{L^{p_n}({B_N})}\lesssim \|f\|_{L^{p}({B_N})}N^{\frac{n}{p_n}-\frac{n}{p}}\;\;\;\text{(by H\"older's inequality)}$$
$$\|f_\theta\|_{L^p(\R^n)}\lesssim N^{\frac{n+1}{2p}-\frac{n+1}{2p_n}}\|f_\theta\|_{L^{p_n}(\R^n)}\;\;\;\text{(by Bernstein's inequality)},$$
and then to let $p\to p_n$.

\section{An $l^p$ decoupling for curves}
In this section we prove Theorem \ref{t2}. It is easy to see that for each $t_0\in [0,1]$, there is an affine transformation $L_{t_0}$ of $\R^n$, more precisely a rotation followed by a translation, such that
$$\begin{cases}L_{t_0}(\Phi(t_0))=\textbf{0} \\ L_{t_0}(\Phi'(t_0))\perp\langle e_2,\ldots,e_n\rangle\\ \hfill L_{t_0}(\Phi''(t_0))\perp\langle e_3,\ldots,e_n\rangle \\  \ldots\ldots\ldots\ldots\ldots\ldots\\  L_{t_0}(\Phi^{(n-1)}(t_0))\perp\langle e_n\rangle \end{cases}$$
In this new local system of coordinates, the equation of the curve near $t=0$ becomes
$$\tilde{\Phi}(t)=(C_{1,1}t+C_{1,2}t^2+\ldots+C_{1,n}t^n,C_{2,2}t^2+\ldots+C_{2,n}t^n,\ldots,C_{n,n}t^n)+O(t^{n+1},t^{n+1},\ldots,t^{n+1}).$$
The coefficients $C_{i,j}$  depend on $t_0$ but satisfy $\kappa\le |C_{i,i}|\le \kappa^{-1}$ and $|C_{i,j}|\le \kappa^{-1}$ for $i<j$, with $\kappa>0$ independent of $t_0$, due to the Wronskian condition.

By invoking a simple induction on scales argument, it suffices to prove Theorem \ref{t2} for the special curves $\Phi_{\textbf{C}}$, $\textbf{C}=(C_{i,j})_{1\le i\le j\le n}$
\begin{equation}
\label{e1}
\Phi_{\textbf{C}}(t)=(C_{1,1}t+C_{1,2}t^2+\ldots+C_{1,n}t^n,C_{2,2}t^2+\ldots+C_{2,n}t^n,\ldots,C_{n,n}t^n),
\end{equation}
with $C_{i,j}$ as above.  The estimates will of course depend only on $\kappa$. To see this, for $\delta<1$, let $K(\delta)$ be the smallest constant such that for each   $f$ with Fourier support in the neighborhood $\A_\delta$ of $\Phi$ we have
$$\|f\|_{2(n+1)}\le K(\delta)(\sum_{\theta\in \P_\delta}\|f_\theta\|_{2(n+1)}^{2(n+1)})^{\frac1{2(n+1)}}.$$

First, for each such  $f$
\begin{equation}
\label{e8}
\|f\|_{2(n+1)}\le K(\delta^{\frac{n}{n+1}})(\sum_{\tau\in \P_{\delta^{\frac{n}{n+1}}}}\|f_\tau\|_{2(n+1)}^{2(n+1)})^{\frac1{2(n+1)}}.
\end{equation}
The previous discussion shows that  the portion of $\Phi$ inside a given $\tau\in \P_{\delta^{\frac{n}{n+1}}}$ is within $\delta$ from a curve \eqref{e1}. Let $a$ be  the left endpoint of the interval of length $\delta^{\frac1{n+1}}$ corresponding to $\tau$. We will perform a simple rescaling as follows.
Consider the linear transformation
$$L_\tau(\xi_1,\ldots,\xi_n)=(\xi_1',\ldots,\xi_n')=(\frac{\xi_1-a}{\delta^{\frac1{n+1}}},\frac{\xi_2-2a\xi_1+a^2}{\delta^{\frac2{n+1}}},\frac{\xi_3-3a\xi_2+3a^2\xi_1-a^3}{\delta^{\frac3{n+1}}},\ldots).$$
It maps $\tau\cap \A_\delta$ into $\A_{\delta^{\frac{1}{n+1}}}$ and each $\theta\in\A_\delta$ with $\theta\subset \tau$ into some $\theta'\in\A_{\delta^{\frac{1}{n+1}}}$. Using this change of variables and Theorem \ref{t2} with $\delta$ replaced with $\delta^{\frac1{n+1}}$ we get
\begin{equation}
\label{e7}
\|f_\tau\|_{2(n+1)}\lesssim_\epsilon\delta^{-\frac1{2(n+1)^2}-\epsilon}(\sum_{\theta\in \P_\delta:\theta\subset\tau}\|f_\theta\|_{2(n+1)}^{2(n+1)})^{\frac1{2(n+1)}}.
\end{equation}
For each $\epsilon>0$, using \eqref{e8} and \eqref{e7} we conclude the existence of $C_\epsilon$ such that for each $\delta<1$
$$K(\delta)\le C_\epsilon\delta^{-\frac1{2(n+1)^2}-\epsilon}K(\delta^{\frac{n}{n+1}}).$$
By iteration this immediately leads to $K(\delta)\lesssim_\epsilon \delta^{-\frac{1}{2(n+1)}-\epsilon}$.

We further observe that it suffices to consider curves $\Phi_{\textbf{C}}$ with $\textbf{C}$ equal to the identity matrix. This is because the decoupling inequality is preserved under linear transformations. For the remainder of the section we let
$$\Phi(t)=(t,t^2,\ldots,t^{n}).$$

For each dyadic interval $I\subset [0,1]$ define the extension operator
$$E_If(x)=\int_If(t)e(tx_1+t^2x_2+\ldots+t^nx_n)dt.$$
Unless specified otherwise, all intervals will be implicitly assumed to be dyadic. It is easy to see that Theorem \ref{t2} is in fact equivalent with the inequality
\begin{equation}
\label{e11}
\|E_{[0,1]}f\|_{L^{2(n+1)}(B_R)}\lesssim_{\epsilon}R^{\frac1{2(n+1)}+\epsilon}(\sum_{\Delta:\frac1{R^{1/n}}-\text{interval }\atop{\Delta\subset [0,1]}}\|E_{\Delta}f\|_{L^{2(n+1)}(w_{B_R})}^{2(n+1)})^{\frac{1}{2(n+1)}}.
\end{equation}

First we prove the multilinear version of \eqref{e11}. This proposition also appears in the companion paper \cite{Bo}, we sketch the details for completeness. We note that this multilinear inequality is more efficient than \eqref{e11}, since it decouples $[0,1]$ into intervals of length $\frac1{R^{1/2}}$, smaller than  $\frac1{R^{1/n}}$. By covering balls of larger radius with balls of smaller radius, we can replace $B_R$ with $B_{R'}$, $R'\ge R$ in Proposition \ref{p2}. If we choose $R'=R^{n/2}$ and then rescale we get
\begin{equation}
\label{e10}
\|(\prod_{j=1}^{n-1}|E_{I_j}f|)^{\frac1{n-1}}\|_{L^{2(n+1)}(B_R)}\le C_{K,\epsilon}R^{\frac{1}{2(n+1)}+\epsilon}(\sum_{\Delta:\frac1{R^{1/n}}-\text{interval }\atop{\Delta\subset [0,1]}}\|E_{\Delta}f\|_{L^{2(n+1)}(w_{B_R})}^{2(n+1)})^{\frac{1}{2(n+1)}}
\end{equation}

Let $K\gg 1$ be a fixed large  constant depending only on $n$, whose value will be chosen at the end of this section.

\begin{proposition}
\label{p2}
Let $I_1,\ldots,I_{n-1}$ be intervals of the form $I_j=[\frac{a_j}{K},\frac{a_j+1}{K}]\subset [\frac1K,\frac{K-1}K]$  with $a_j\le a_{j+1}-2$. Then for each $R$ large enough, each $R$-ball $B_R$  and each $\epsilon>0$ there exists $C_{K,\epsilon}$ such that
$$\|(\prod_{j=1}^{n-1}|E_{I_j}f|)^{\frac1{n-1}}\|_{L^{2(n+1)}(B_R)}\le C_{K,\epsilon}R^{\frac{n}{4(n+1)}+\epsilon}(\sum_{\Delta:\frac1{R^{1/2}}-\text{interval }\atop{\Delta\subset [0,1]}}\|E_{\Delta}f\|_{L^{2(n+1)}(w_{B_R})}^{2(n+1)})^{\frac{1}{2(n+1)}}.$$
\end{proposition}

\begin{proof}
The first stage of the argument is concerned with proving  that the system
\begin{equation}
\label{e4}
\begin{cases}\xi_1=t_1+\ldots+t_{n-1} \\ \xi_2=t_1^2+\ldots+t_{n-1}^2\\\dots\ldots \dots\ldots\dots\ldots\\ \xi_{n}=t_1^{n}+\ldots+t_{n-1}^n\end{cases}
\end{equation}
with $$0<t_1<t_2<\ldots<t_{n-1}\le 1$$
gives rise to a hypersurface $S:=\{\xi_1=H(\xi_2,\ldots,\xi_n)\}$ with nonzero Gaussian curvature. To see this, fix a point $\xi=(\xi_1,\ldots,\xi_n)\in S$ corresponding to $(t_1,\ldots,t_{n-1})$. Consider also a point $\xi'=(\xi_1',\ldots,\xi_n')\in S$ corresponding to $(t_1+s_1,\ldots,t_{n-1}+s_{n-1})$ with small $s_i$. Denote by $\1$ the vector in $\R^{n-1}$ with all entries equal to 1 and by $s=(s_1,\ldots,s_{n-1})$.  Then using Taylor's formula we get
\begin{equation}
\label{e2}
\left(\begin{array}{ccc}\xi_2'-\xi_2\\\xi_3'-\xi_3\\\ldots\\\xi_n'-\xi_{n}\end{array}\right)=D_1\left(\begin{array}{ccc}s_1\\s_2\\\ldots\\s_{n-1}\end{array}\right)+\frac12D_2\left(\begin{array}{ccc}s_1^2\\s_2^2\\\ldots\\s_{n-1}^2\end{array}\right)+\1 O(\|s\|^3)
\end{equation}
where
$$D_1=\begin{bmatrix} 2t_1&\ldots & 2t_{n-1}&\\3t_1^2&\ldots& 3t_{n-1}^2&\\ \ldots& \ldots&\ldots&\\nt_1^{n-1}&\ldots & n t_{n-1}^{n-1}&  \end{bmatrix}\text{ and }D_2=\begin{bmatrix} 2&\ldots & 2&\\6t_1&\ldots& 6t_{n-1}&\\ \ldots& \ldots&\ldots&\\n(n-1)t_1^{n-2}&\ldots & n(n-1) t_{n-1}^{n-2}&  \end{bmatrix}.$$
Out hypothesis guarantees that
$D_1$ is non-singular and thus we can write
\begin{equation}
\label{e3}
\xi_1'-\xi_1=\langle D_1^{-1}\left(\begin{array}{ccc}\xi_2'-\xi_2\\\ldots\\\xi_n'-\xi_{n}\end{array}\right),\1\rangle-\frac12\sum_{j=1}^{n-1}s_j^2\langle D_1^{-1}D_2e_j,\1\rangle+ O(\|s\|^3).
\end{equation}
With
$$\left(\begin{array}{ccc}\eta_1\\\ldots\\\eta_{n-1}\end{array}\right)=D_1^{-1}\left(\begin{array}{ccc}\xi_2'-\xi_2\\\ldots\\\xi_n'-\xi_{n}\end{array}\right)$$
equation \eqref{e2} implies that
$$s_i=\eta_i+O(\|s\|^2).$$
This in turn shows that $\eta_i=O(\|s\|)$ and thus
$$s_i^2=\eta_i^2+O(\|s\|^3).$$
We conclude using \eqref{e3} that
$$\xi_1'-\xi_1=\sum_{j=1}^{n-1}\eta_j-\frac12\sum_{j=1}^{n-1}\eta_j^2\langle D_1^{-1}D_2e_j,\1\rangle+ O(\|\eta\|^3).$$
It remains to prove that $\langle D_1^{-1}D_2e_j,\1\rangle\not =0$. This follows from a standard application of the mean value theorem from Calculus, see \cite{Bo} for details.  The  first stage of the argument is now complete.

Next we let $S_0$ be the part of the hypersurface \eqref{e4} corresponding to $t_j\in I_j$. The entries of $D_1^{-1}$ are $O_K(1)$ and so the volume element $dV=\rho(t_1,\ldots,t_{n-1})dt_1\ldots dt_{n-1}$ on $S_0$ satisfies
\begin{equation}
\label{e5}
1\le \rho\lesssim_K 1.
\end{equation}
 Thus
$$\prod_{j=1}^{n-1}E_{I_j}f(x)=\int_{S_0}F(\xi)e(\xi\cdot x)d\sigma(\xi)$$
where, if $\xi=(\xi_1,\ldots,\xi_{n})\in S_0$ corresponds to the parameters $(t_1,\ldots,t_{n-1})$ then
$$F(\xi_1,\ldots,\xi_{n})=\frac{\prod_{i=1}^{n-1}f(t_i)}{\rho(t_1,\ldots,t_{n-1})}.$$
Due to \eqref{e5}, for each  intervals $\Delta_j\subset I_j$ of length $R^{-1/2}$ the cap on $S_0$ corresponding to $t_j\in \Delta_j$ has diameter $\lesssim_K R^{-1/2}$. Also, since the entries of $D_1^{-1}$ are $O_K(1)$, the caps corresponding to distinct choices of $(\Delta_1,\ldots,\Delta_{n-1})$ are separated by $\gtrsim_K R^{-1/2}$.

We can now invoke Theorem \ref{t1} followed by H\"older's inequality to get
$$\|(\prod_{j=1}^{n-1}|E_{I_j}f|)^{\frac1{n-1}}\|_{L^{2(n+1)}(B_R)}=\|\widehat{Fd\sigma_{S_0}}\|_{L^{\frac{2(n+1)}{n-1}}(B_R)}^{\frac{1}{n-1}}\lesssim_{\epsilon,K} $$

$$R^{\frac14-\frac{n-1}{4(n+1)}+\epsilon}(\sum_{\Delta_{j_1}\subset I_1\ldots \Delta_{j_{n-1}}\subset I_{n-1}}\|\prod_{i=1}^{n-1}E_{\Delta_{j_i}}f\|_{L^{\frac{2(n+1)}{n-1}}(w_{B_R})}^{\frac{2(n+1)}{n-1}})^{\frac1{2(n+1)}}\le$$

$$R^{\frac14-\frac{n-1}{4(n+1)}+\epsilon}(\prod_{i=1}^{n-1}\sum_{\Delta_{j_i}\subset I_i}\|E_{\Delta_{j_i}}f\|_{L^{2(n+1)}(w_{B_R})}^{\frac{2(n+1)}{n-1}})^{\frac1{2(n+1)}}\le $$

$$R^{\frac{n}{4(n+1)}+\epsilon}\prod_{i=1}^{n-1}(\sum_{\Delta_{j_i}\subset I_i}\|E_{\Delta}f\|_{L^{2(n+1)}(w_{B_R})}^{2(n+1)})^{\frac{1}{2(n+1)(n-1)}}\le $$$$R^{\frac{n}{4(n+1)}+\epsilon}(\sum_{\Delta:\frac1{R^{1/2}}-\text{interval }\atop{\Delta\subset [0,1]}}\|E_{\Delta}f\|_{L^{2(n+1)}(w_{B_R})}^{2(n+1)})^{\frac{1}{2(n+1)}}.$$
\end{proof}
It is easy to see that when $n=3$, the surface $S$ parametrized by \eqref{e4} is $$\xi_3=-\frac{\xi_1^3}{2}+3\xi_1\xi_2.$$ The principal curvatures are $-\frac32(\xi_1\pm\sqrt{\xi_1^2+1})$ and the second fundamental form of $S$ is not definite. This example explains the need for our main Theorem \ref{t1} in this paper. 
\bigskip

Next we employ the Bourgain-Guth induction on scales \cite{BG} to bound the linear operator by multilinear ones.
\begin{proposition}
\label{p3}
For each $\epsilon>0$, for each $R$ large enough and each $R$-ball $B_R$ we have
$$\|E_{[0,1]}f\|_{L^{2(n+1)}(B_R)}\le C_{K,\epsilon}R^{\frac{1}{2(n+1)}+\epsilon}(\sum_{\Delta:\frac1{R^{1/n}}-\text{interval }\atop{\Delta\subset [0,1]}}\|E_{\Delta}f\|_{L^{2(n+1)}(w_{B_R})}^{2(n+1)})^{\frac{1}{2(n+1)}}+$$
$$+C(\sum_{I:\frac1K-\text{interval}}\|E_{I}f\|_{L^{2(n+1)}(B_R)}^{2(n+1)})^{\frac{1}{2(n+1)}}.$$
The constant $C$ is independent of $K,\epsilon$, it only depends on $n$.
\end{proposition}
\begin{proof}
We start by writing
$$E_{[0,1]}f=\sum_{I:\frac1K-\text{interval }}E_If.$$
We call the collection of $\frac1K$-intervals $I_1,\ldots, I_{n-1}$ transverse if they satisfy the requirement in the statement of Proposition \ref{p2}. It is rather immediate that for each $x\in B_R$
$$|E_{[0,1]}f(x)|\le C\max_{I:\frac1K-\text{interval}}|E_If(x)|+C_K\sum_{I_1,\ldots,I_{n-1}:\text{transverse}}(\prod_{j=1}^{n-1}|E_{I_j}f(x)|)^{\frac1{n-1}}\le$$
$$C(\sum_{I:\frac1K-\text{interval}}|E_If(x)|^{2(n+1)})^{\frac1{2(n+1)}}+C_K\sum_{I_1,\ldots,I_{n-1}:\text{transverse}}(\prod_{j=1}^{n-1}|E_{I_j}f(x)|)^{\frac1{n-1}}.$$
The result now follows by integrating the $2(n+1)$ power and by using \eqref{e10}.
\end{proof}
We now rescale to get the following version.
\begin{proposition}
\label{p4}
Let $I$ be a $\delta$- interval in $[0,1]$.
For each $\epsilon>0$, for each $R\gtrsim \delta^{-n}$ and each $R$-ball $B_R$ we have
$$\|E_{I}f\|_{L^{2(n+1)}(B_R)}\le C_{K,\epsilon}(R\delta^n)^{\frac{1}{2(n+1)}+\epsilon}(\sum_{\Delta:\frac1{R^{1/n}}-\text{interval }\atop{\Delta\subset I}}\|E_{\Delta}f\|_{L^{2(n+1)}(w_{B_R})}^{2(n+1)})^{\frac{1}{2(n+1)}}+$$
$$+C_n(\sum_{J:\frac\delta K-\text{interval}\atop{J\subset I}}\|E_{J}f\|_{L^{2(n+1)}(B_R)}^{2(n+1)})^{\frac{1}{2(n+1)}}.$$
The constants $C_{K,\epsilon},C_n$ do not depend on $\delta$ and $C_n$ is independent of $K,\epsilon$ (it only depends on $n$).
\end{proposition}
\begin{proof}
Note that if $I=[a,a+\delta]$ then the change of variables $s=\frac{t-a}{\delta}$ shows that
$$|E_If(x)|=\delta|E_{[0,1]}f^{a,\delta}(x')|$$
where $f^{a,\delta}(s)=f(s\delta+a)$ and $x'=(x_1',\ldots,x_n')$ with
$$\begin{cases}x_1'=\delta(x_1+2ax_2+3a^2x_3+\ldots)\\ x_2'=\delta^2(x_2+3ax_3+\ldots)\\\ldots\ldots\ldots\\x_n'=\delta^nx_n
\end{cases}.$$
In particular
$$\|E_If\|_{L^{2(n+1)}(B_R)}=\delta^{1-\frac{n}4}\|E_{[0,1]}f^{a,\delta}\|_{L^{2(n+1)}(C_R)}$$
where $C_R$ is a $\sim \delta R\times \delta^2 R\times \ldots \times \delta^n R$- cylinder. Cover $C_R$ by $O(1)$-overlapping $\delta^nR$- balls $B_{\delta^nR}$ and write using Proposition \ref{p3}
$$\|E_If\|_{L^{2(n+1)}(B_R)}\lesssim \delta^{1-\frac{n}4}(\sum_{B_{\delta^nR}}\|E_{[0,1]}f^{a,\delta}\|_{L^{2(n+1)}(B_{\delta^nR})}^{2(n+1)})^{\frac1{2(n+1)}}\le$$

$$C_{K,\epsilon}(\delta^nR)^{\frac{1}{2(n+1)}+\epsilon}\delta^{1-\frac{n}4}(\sum_{B_{\delta^nR}}\sum_{\Delta:\frac1{\delta R^{1/n}}-\text{interval }\atop{\Delta\subset [0,1]}}\|E_{\Delta}f^{a,\delta}\|_{L^{2(n+1)}(w_{B_{\delta^n R}})}^{2(n+1)})^{\frac{1}{2(n+1)}}+$$

$$+C\delta^{1-\frac{n}4}(\sum_{B_{\delta^nR}}\sum_{H:\frac1K-\text{interval}}\|E_{H}f^{a,\delta}\|_{L^{2(n+1)}(B_{\delta^n R})}^{2(n+1)})^{\frac{1}{2(n+1)}}\lesssim$$
$$C_{K,\epsilon}(\delta^nR)^{\frac{1}{2(n+1)}+\epsilon}\delta^{1-\frac{n}4}(\sum_{\Delta:\frac1{\delta R^{1/n}}-\text{interval }\atop{\Delta\subset [0,1]}}\|E_{\Delta}f^{a,\delta}\|_{L^{2(n+1)}(w_{C_R})}^{2(n+1)})^{\frac{1}{2(n+1)}}+$$

$$+C\delta^{1-\frac{n}4}(\sum_{H:\frac1K-\text{interval}}\|E_{H}f^{a,\delta}\|_{L^{2(n+1)}(w_{C_{R}})}^{2(n+1)})^{\frac{1}{2(n+1)}}.$$
Changing back to the original variables gives us the desired estimate.
\end{proof}
\bigskip

We are now ready to prove \eqref{e11}. Choose $K$ large enough so that $2C_n\le K^n$, where $C_n$ is the constant in Proposition \ref{p4}. Iterate Proposition \ref{p4} starting with scale $\delta=1$ until we reach scale $\delta=R^{-1/n}$.
Each iteration  lowers the scale of the intervals from $\delta$ to  $\frac{\delta}{K}$. We thus have to iterate $\frac{\log_K R}{n}$ times. In particular
$$\|E_{[0,1]}f\|_{L^{2(n+1)}(B_R)}\lesssim\sum_{s=0}^{\frac{\log_K R}{n}} C_{K,\epsilon}(C_n)^s(RK^{-ns})^{\frac{1}{2(n+1)}+\epsilon}(\sum_{\Delta:\frac1{R^{1/n}}-\text{interval }\atop{\Delta\subset [0,1]}}\|E_{\Delta}f\|_{L^{2(n+1)}(w_{B_R})}^{2(n+1)})^{\frac{1}{2(n+1)}}$$
$$\le C_{K,\epsilon}R^{\frac1{2(n+1)}+\epsilon}(\sum_{\Delta:\frac1{R^{1/n}}-\text{interval }\atop{\Delta\subset [0,1]}}\|E_{\Delta}f\|_{L^{2(n+1)}(w_{B_R})}^{2(n+1)})^{\frac{1}{2(n+1)}}.$$


\begin{thebibliography}{99}



\bibitem{BG} Bourgain, J. and  Guth, L. {\em Bounds on oscillatory integral operators based on multilinear estimates} Geom. Funct. Anal. 21 (2011), no. 6, 1239-1295
\bibitem{BD3} Bourgain, J. and Demeter, C. {\em The proof of the $l^2$ Decoupling Conjecture}, preprint available on arxiv.
\bibitem{Bo} Bourgain, J. {\em Decoupling inequalities and some mean-value theorems}, preprint available on arxiv.
\bibitem{Hu} Hua, L. K. {\em Additive theory of prime numbers. Translations of Mathematical Monographs,} Vol. 13 American Mathematical Society, Providence, R.I. 1965 xiii+190 pp.
\bibitem{TWol} Wolff, T. {\em Local smoothing type estimates on $L^p$ for large $p$}, Geom. Funct. Anal. 10 (2000), no. 5, 1237-1288
\end{thebibliography}
\end{document}